\documentclass[a4paper,12pt]{amsart}
\usepackage{amssymb}
\usepackage{ifthen}
\usepackage{graphicx}
\usepackage{float}
\usepackage{caption}
\usepackage{subcaption}
\usepackage{cite}
\usepackage{amsfonts}
\usepackage{amscd}
\usepackage{amsxtra}
\usepackage{color}
\usepackage[dvipsnames]{xcolor}

\newtheorem{theorem}{Theorem}[section]
\newtheorem{lemma}[theorem]{Lemma}

\newtheorem{conj}{Conjecture}
\newtheorem{prob}{Problem}

\theoremstyle{definition}
\newtheorem{defn}{Definition}[section]
\newtheorem{example}{Example}

\newcounter{alphabet}


\def\be{\begin{equation}}
\def\ee{\end{equation}}

\newcommand{\bee}{\begin{enumerate}}
\newcommand{\eee}{\end{enumerate}}

\newcommand{\blem}{\begin{lem}}
\newcommand{\elem}{\end{lem}}
\newcommand{\bthm}{\begin{thm}}
\newcommand{\ethm}{\end{thm}}
\newcommand{\bcor}{\begin{cor}}
\newcommand{\ecor}{\end{cor}}
\newcommand{\beg}{\begin{example}}
\newcommand{\eeg}{\end{example}}
\newcommand{\begs}{\begin{examples}}
\newcommand{\eegs}{\end{examples}}
\newcommand{\bdefe}{\begin{defn}}
\newcommand{\edefe}{\end{defn}}
\newcommand{\bprob}{\begin{prob}}
\newcommand{\eprob}{\end{prob}}
\newcommand{\bques}{\begin{ques}}
\newcommand{\eques}{\end{ques}}
\newcommand{\bei}{\begin{itemize}}
\newcommand{\eei}{\end{itemize}}
\newcommand{\bcon}{\begin{conj}}
\newcommand{\econ}{\end{conj}}
\newcommand{\bcons}{\begin{conjs}}
\newcommand{\econs}{\end{conjs}}
\newcommand{\bprop}{\begin{propo}}
\newcommand{\eprop}{\end{propo}}
\newcommand{\br}{\begin{rem}}
\newcommand{\er}{\end{rem}}
\newcommand{\brs}{\begin{rems}}
\newcommand{\ers}{\end{rems}}
\newcommand{\bo}{\begin{obser}}
\newcommand{\eo}{\end{obser}}
\newcommand{\bos}{\begin{obsers}}
\newcommand{\eos}{\end{obsers}}
\newcommand{\bpf}{\begin{proof}}
\newcommand{\epf}{\end{proof}}
\newcommand{\ba}{\begin{array}}
\newcommand{\ea}{\end{array}}
\newcommand{\beq}{\begin{align}}
\newcommand{\beqq}{\begin{align*}}
\newcommand{\eeq}{\end{align}}
\newcommand{\eeqq}{\end{align*}}

\newcounter{minutes}
\divide\time by 60
\newcounter{hours}
\multiply\time by 60 \addtocounter{minutes}{-\time}

\begin{document}
\bibliographystyle{amsplain}

\title[]%
{Generalize Hilbert operator acting on Dirichlet spaces}

\author{L.Y. Zhao Z.Y. Wang* and Z.R. Su
}
\begin{abstract}
Let  $\mu$ be a positive Borel measure on the interval $[0,1)$. For $\gamma>0$, the Hankel matrix $\mathcal{H}_{\mu,\gamma}=(\mu_{n,k})_{n,k\geq0}$ with entries $\mu_{n,k}=\mu_{n+k}$, where $\mu_{n+k}=\int_{0}^{\infty}t^{n+k}d\mu(t)$. formally induces the operator $$\mathcal{H}_{\mu,\gamma}=\sum_{n=0}^{\infty}\left(\sum_{k=0}^{\infty}\mu_{n,k}a_k\right)\frac{\Gamma(n+\gamma)}{n!\Gamma(\gamma)}z^n,$$
on the space of all analytic functions $f(z)=\sum_{k=0}^{\infty}{a_k}{z^k}$ in the unit disc $\mathbb{D}$. Following  ideas from \cite{author3} and \cite{author4}, in this paper, for $0\leq\alpha<2$, $2\leq\beta<4$, $\gamma\geq1$, We characterize the measure $\mu$ for which $\mathcal{H}_{\mu,\gamma}$ is bounded(resp.,compact)from $\mathcal{D}_{\alpha}$ into $\mathcal{D}_{\beta}$.
\\\hspace*{\fill}\\
\textit{Keywords:}
Generalized Hilbert operator,  Dirichlet spaces, Carleson measure
\end{abstract}

\thanks{\ \ *Corresponding author. }
\thanks{\ \ \ \textit{Email addresses:}\texttt{1914407155@qq.com}(L.Y. Zhao),\ \
\texttt{zywang@gdut.edu.cn}(Z.Y. Wang),\\ \texttt{1131258896@qq.com}(Z.R. Su).\\
\textit{Addresses:Department of Mathematics and Statistics,Guangdong University of Technology,
510520 Guangzhou, Guangdong, P. R. China.}}

\maketitle

\section{Introduction}
Let $\mathbb{D}=\{z\in\mathbb{C}:|z|<1\}$ denote the open unit  disc in the complex plane $\mathbb{C}$ and let ${H}(\mathbb{D})$ be the space of all analytic functions in $\mathbb{D}$. For $f\in H (\mathbb{D})$ and $0<r<1$, the integral mean $M_p(r,f)$ is defined by $$M_p(r,f)=\left\{\frac{1}{2\pi}\int_{0}^{2\pi}|f(re^{i\theta})|^pd\theta\right\}^{\frac{1}{p}}.$$
For $0<p<\infty$, the Hardy space $H^p$ is the set of $f\in H(\mathbb{D})$ with
$$\|f\|_{H^p}=\sup_{0<r<1}M_p(r,f)<\infty.$$
See  more detail information about Hardy space spaces in \cite{author1}.
The Dirichlet space $\mathcal{D}_{\alpha}$, $\alpha\in R$, consists of functions $f(z)=\sum_{n=0}^{\infty}a_nz^n\in H(\mathbb{D})$ for which
$$\|f\|_{\mathcal{D}_{\alpha}}^2=\sum_{n=0}^{\infty}(n+1)^{1-\alpha}|a_n|^2<\infty.$$
When $\alpha=0$, $\mathcal{D}_{0}$ is the classical Dirichlet space $\mathcal{D}$, When When $\alpha=1$, $\mathcal{D}_{1}$ is the Hardy space $H^2$. See  more detail theory about Dirichlet space spaces in \cite{author1,author8}.\\

\indent
For an arc ${I}\subseteq\partial\mathbb{D}$, let $|I|=\frac{1}{2\pi}\int_{I}|d\xi|$ be the normalized length of ${I}$ and ${S}({I})$ be the Carleson square based on $I$ with
$$ {S}({I})=\{z=re^{it}:e^{it}\in {I};1-|I|\leq r<1\}.$$
Clearly, if ${I}=\partial\mathbb{D}$, then ${S}({I})=\mathbb{D}$.\\
\indent
For $0<s<\infty$, we have that if a positive Borel measure on $\mathbb{D}$ satisfies
$$\sup_{I\subset\partial\mathbb{D}}\frac{\mu(S(I))}{|I|^s}<\infty,$$ then it become a $s$-Carleson measure(See \cite{author9}).
If $s=1$, the $1$-Carleson measure is the classical Carleson measure. Suppose that $\mu$ is a positive Borel measure on $\mathbb{D}$, if the following equation
$$\lim_{|I|\to 0}\frac{\mu(S(I))}{|I|^s}=0$$ is true, then
$\mu$ is a vanishing $s$-Carleson measure.
If $s=1$, the vanishing $1$-Carleson measure is the vanishing  Carleson measure.\\
 \indent
We see a positive Borel measure $\mu$ on $[0,1)$  as a Borel measure on $\mathbb{D}$
by identifying it with the measure $\tilde{\mu}$ which  defined by
$$\tilde{\mu}(\textit E)=\mu(\textit E\cap[0,1)),$$
for any Borel subset $E$ of $\mathbb{D}$. Then a positive Borel measure $\mu$ on $[0,1)$
can be seen as an $s$-Carleson measure on $\mathbb{D}$, if
$$\sup_{t\in[0,1)} \frac{\mu([t,1))}{(1-t)^s}<\infty.$$
For vanishing $s$-Carleson measure, we have similar definition.\\
\indent
If $\mu$ is a positive Borel measure on the interval $[0,1)$,
we let $\mathcal{H}_{\mu}$ be the  Hankel matrix $\mathcal{H}_{\mu}=(\mu_{n,k})_{n,k\geq0}$
with entries $\mu_{n,k}=\mu_{n+k}$ and $\mu_n=\int_{[0,1)}t^{n}d\mu(t)$.

\indent
For any analytic functions $f(z)=\sum_{n=0}^{\infty}a_nz^n$.
In \cite{author4,author5}, Ye and Zhou first defined the Derivative-Hilbert operator  as
$$\mathcal{DH}_{\mu}(f)(z)=\sum_{n=0}^{\infty}\left(\sum_{k=0}^{\infty}\mu_{n,k}a_k\right)(n+1)z^n,\ z\in \mathbb{D},$$
on the space of analytic functions in $\mathbb{D}$.
In \cite{author2}, for $\gamma>0$, Ye and Zhou first defined the generalized Hilbert operator $\mathcal{H}_{\mu,\gamma}$ as
$$\mathcal{H}_{\mu,\gamma}(f)(z)=\sum_{n=0}^{\infty}\left(\sum_{k=0}^{\infty}\mu_{n,k}a_k\right)\frac{\Gamma(n+\gamma)}{n!\Gamma(\gamma)}z^n,\ z\in \mathbb{D},$$
on the space of analytic functions in $\mathbb{D}$.
It is clear that when $\gamma=2$, $\mathcal{H}_{\mu,\gamma}=\mathcal{DH}_{\mu}.$\\
\indent
For $0<\alpha<2$, $2\leq\beta<4$, in \cite{author3}, Xu characterized the measure $\mu$ for which $\mathcal{DH}_{\mu}$ is bounded(resp.,compact)from $\mathcal{D}_{\alpha}$ into $\mathcal{D}_{\beta}$. In this paper, We find using the same ways in \cite{author3}, we obtain the necessary and sufficient condition of $\mu$ such that $\mathcal{H}_{\mu,\gamma}$ is bounded(resp.,compact)from $\mathcal{D}_{\alpha}$ into $\mathcal{D}_{\beta}$.\\
Notation. Throughout this paper, $C$ denotes a positive constant which may be different from one occurrence to the next. The symbol $E_1\preceq E_2$ means that there exists a positive constant $C$ such that $E_1\leq C E_2$. If $E_1\preceq E_2$ and $E_2\preceq E_1$  set up at the same time, we denote by $E_1\asymp E_2$.
\section{main results}
\begin{lemma}\label{2.1}
Suppose that $\alpha\in \mathbb{R}$, $\gamma>0$ and let $\mu$ be a positive Borel measure on $[0,1)$. If $\mu_n=\mathcal{O}(n^{-(\frac{\alpha}{2}+\varepsilon)})$ is hold for some $\varepsilon>0$, then $\mathcal{H}_{\mu,\gamma}$ is well defined on $\mathcal{D}_{\alpha}$.
\end{lemma}

\begin{proof}
In \cite{author3}, Xu proved that  if $\mu_n=\mathcal{O}(n^{-(\frac{\alpha}{2}+\epsilon)})$ is hold for some $\varepsilon>0$, then $|\sum_{n=0}^{\infty}\mu_{n,k}a_k|<\infty$ for any $f(z)=\sum_{n=0}^{\infty}a_nz^n\in \mathcal{D}_{\alpha}$, which implies that $\mathcal{H}_{\mu,\gamma}$ is well defined on $\mathcal{D}_{\alpha}$.
\end{proof}
The following lemma plays an important role in this paper(see \cite{author6}).
\begin{lemma}\label{2.2}
Let K(x,y) be a real function of two variables such that \\
 $(i)$ $K(x,y)$ is a positive and homogenous with degree $-1$.\\
 $(ii)$ $$\int_{0}^{\infty}K(x,1)x^{-\frac{1}{2}}=\int_{0}^{\infty}K(1,y)y^{-\frac{1}{2}}dy=k.$$
 $(iii)$ The functions $K(x,1)x^{-\frac{1}{2}}$ and $K(1,y)y^{-\frac{1}{2}}$ are strictly decreasing.\\
 Then for every sequence $\{a_n\}_{n\geq0}$ in $l^2$ it is $$\sum_{n=1}^{\infty}\left(\sum_{k=1}^{\infty}K(n,k)a_k\right)^2\leq k^2\sum_{n=1}^{\infty}a_n^2.$$

\end{lemma}

\begin{lemma}\label{2.3}
Suppose that $0<\alpha\leq2$, $2\leq\beta<4$, $\gamma\geq1$ and let $\mu$ be a positive Borel measure on $[0,1)$ which satisfies the condition in Lemma $2.1$. Then the following two conditions are
equivalent:\\
\indent $(i)$  $\mu$ is a $(\gamma-\frac{\beta-\alpha}{2})$-Carleson measure.\\
\indent $(ii)$  $\mathcal{H}_{\mu,\gamma}$ is bounded  from $\mathcal{D}_{\alpha}$ into $\mathcal{D}_{\beta}$.\\
\end{lemma}
\begin{proof}
$(i)\Rightarrow(ii)$. First, for $\alpha\in \mathbb{R}$ we defined two map $V_{\alpha}$ and $V_{\alpha}^{-1}$. \\
For $f(z)=\sum_{n=0}^{\infty}a_nz^n\in \mathcal{D}_{\alpha}$, the map $$V_{\alpha}(f)(z)=\sum_{n=0}^{\infty}(n+1)^{(1-\alpha)/2}a_nz^n,$$
Since $$\|V_{\alpha}(f)\|_{H^2}^2=\sum_{n=0}^{\infty}(n+1)^{1-\alpha}|a_n|^2=\|f\|_{\mathcal{D}_{\alpha}}^2,$$
$V_{\alpha}$ is a a bounded operator from $\mathcal{D}_{\alpha}$ into $H^2$.\\
For $g(z)=\sum_{n=0}^{\infty}b_n z^n\in H^2$, the map $$V_{\alpha}^{-1}(g)(z)=\sum_{n=0}^{\infty}(n+1)^{(1-\alpha)/2}b_nz^n,$$
Since
$$\|V_{\alpha}^{-1}(g)\|_{\mathcal{D}_{\alpha}}^2=\sum_{n=0}^{\infty}|b_n|^2=\|g\|_{H^2}^2,$$
$V_{\alpha}^{-1}$ is a a bounded operator from $H^2$ into $\mathcal{D}_{\alpha}$.\\
Suppose that $\mu$ is a $(\gamma-\frac{\beta-\alpha}{2})$-Carleson measure, from Theorem $2.1$ of \cite{author7}, we know that $$\mu_n=\mathcal{O}\left({n^{-\gamma+\frac{\beta-\alpha}{2}}}\right).\eqno{(2.1)}$$
Since $V_{\alpha}$,$V_{\alpha}^{-1}$ are bounded operators, from \cite{author8}, we know  that the generalized operator $\mathcal{H}_{\mu,\gamma}$ is bounded from $\mathcal{D}_{\alpha}$ into $\mathcal{D}_{\beta}$ if and only if the operator $V_{\beta}\circ\mathcal{H}_{\mu,\gamma}\circ V_{\alpha}^{-1}$ is bounded on Hardy space $H^2.$ \\
For $f(z)=\sum_{n=0}^{\infty}a_nz^n\in {H}^2$, we define a new operator $S_{\mu,\gamma}$ as $$S_{\mu,\gamma}(f)(z)=V_{\beta}\circ\mathcal{H}_{\mu,\gamma}\circ V_{\alpha}^{-1}(f)(z)=\sum_{n=0}^{\infty}\left(\sum_{k=0}^{\infty}(n+1)^{\frac{1-\beta}{2}}(k+1)^{\frac{\alpha-1}{2}}\mu_{n+k}a_k\right)\frac{\Gamma(n+\gamma)}{n!\Gamma(\gamma)}z^n.$$
Notice that
$$\frac{\Gamma(n+\gamma)}{n!\Gamma(\gamma)}\asymp(n+1)^{\gamma-1},\eqno{(2.2)}$$
we have
$$S_{\mu,\gamma}(f)(z)\asymp\sum_{n=0}^{\infty}\left(\sum_{k=0}^{\infty}(n+1)^{\frac{2\gamma-\beta-1}{2}}
(k+1)^{\frac{\alpha-1}{2}}\mu_{n+k}a_k\right)z^n.$$
with $(2.1)$, we have
\begin{align*}
\|S_{\mu,\gamma}(f)(z)\|_{H^2}^{2}&\asymp\sum_{n=0}^{\infty}\left|\sum_{k=0}^{\infty}(n+1)^{\frac{2\gamma-\beta-1}{2}}(k+1)^{\frac{\alpha-1}{2}}\mu_{n,k}a_k\right|^2\\
&\leq\sum_{n=0}^{\infty}\left(\sum_{k=0}^{\infty}(n+1)^{\frac{2\gamma-\beta-1}{2}}(k+1)^{\frac{\alpha-1}{2}}\mu_{n,k}|a_k|\right)^2\\
&\lesssim\sum_{n=0}^{\infty}\left(\sum_{k=0}^{\infty}(n+1)^{\frac{2\gamma-\beta-1}{2}}(k+1)^{\frac{\alpha-1}{2}}\frac{|a_k|}{(n+k+2)^{\gamma-\frac{\beta-\alpha}{2}}}\right)^2\\
&=\sum_{n=1}^{\infty}\left(\sum_{k=1}^{\infty}n^{\frac{2\gamma-\beta-1}{2}}k^{\frac{\alpha-1}{2}}\frac{|a_k|}{(n+k)^{\gamma-\frac{\beta-\alpha}{2}}}\right)^2\\
\end{align*}
Let $$K(x,y)=x^{\frac{2\gamma-\beta-1}{2}}y^{\frac{\alpha-1}{2}}\frac{1}{(x+y)^{\gamma-\frac{\beta-\alpha}{2}}},x>0,y>0.$$
Then we have
$$\int_{0}^{\infty} K(x,1)x^{-\frac{1}{2}}dx=\int_{o}^{\infty}\frac{x^{\gamma-\frac{\beta}{2}-1}}{(x+1)^{\gamma-\frac{\beta-\alpha}{2}}}dx=B(\gamma-\frac{\beta}{2},\frac{\alpha}{2}),$$
and

$$\int_{0}^{\infty} K(1,y)y^{-\frac{1}{2}}dy=\int_{o}^{\infty}\frac{y^{\frac{\alpha}{2}-1}}{(y+1)^{\gamma-\frac{\beta-\alpha}{2}}}dy=B(\frac{\alpha}{2},\gamma-\frac{\beta}{2}).$$
We know that Beta function have
symmetry, $B(\gamma-\frac{\beta}{2},\frac{\alpha}{2})=B(\frac{\alpha}{2},\gamma-\frac{\beta}{2})$.\\ Since $\int_{0}^{\infty}K(x,1)x^{-\frac{1}{2}}$ and $\int_{0}^{\infty} K(1,y)y^{-\frac{1}{2}}$ are strictly decreasing, then  Lemma $2.2$ gives that
$$\sum_{n=1}^{\infty}\left(\sum_{k=1}^{\infty}n^{\frac{2\gamma-\beta-1}{2}}k^{\frac{\alpha-1}{2}}\frac{|a_k|}{(n+k)^{\gamma-\frac{\beta-\alpha}{2}}}\right)^2\lesssim\left(B(\gamma-\frac{\beta}{2},\frac{\alpha}{2})\right)^2\|f\|_{H^2}^2,\eqno{(2.3)}$$
which implies that the operator $S_{\mu,\gamma}$ is bounded on $H^2$,
then $\mathcal{H}_{\mu,\gamma}$ is bounded from $\mathcal{D}_{\alpha}$ into $\mathcal{D}_{\beta}$.\\
$(ii)\Rightarrow(i)$. For $0<t<1$, let $f_t(z)=(1-t^2)^{1-\frac{\alpha}{2}}\sum_{n=0}^{\infty}t^nz^n$. Since $$\|f_t\|_{\mathcal{D}_{\alpha}}^2=(1-t^2)^{2-\alpha}\sum_{n=0}^{\infty}(n+1)^{1-\alpha}t^{2n}\asymp1$$
Using $(2.2)$, we have
\begin{align*}
\|\mathcal{H}_{\mu,\gamma}(f_t)\|_{\mathcal{D}_{\beta}}^{2}
&=\sum_{n=0}^{\infty}(n+1)^{1-\beta}\left(\sum_{k=0}^{\infty}\mu_{n+k}(1-t^2)^{1-\frac{\alpha}{2}}\frac{\Gamma(n+\gamma)}{n!\Gamma(\gamma)}\right)^2\\
&\asymp\sum_{n=0}^{\infty}(n+1)^{1-\beta}\left(\sum_{k=0}^{\infty}(n+1)^{\gamma-1}\mu_{n+k}(1-t^2)^{1-\frac{\alpha}{2}}\right)^2\\
&\succeq(1-t^2)^{2-\alpha}\sum_{n=0}^{\infty}(n+1)^{2\gamma-\beta-1}\left(\sum_{k=0}^{\infty}t^k\int_{t}^1x^{n+k}d\mu(x)\right)^2\\
&\succeq(1-t^2)^{2-\alpha}\sum_{n=0}^{\infty}(n+1)^{2\gamma-\beta-1}\left(\sum_{k=0}^{\infty}t^{n+2k}\mu([t,1))\right)^2
\end{align*}
Since $\mathcal{H}_{\mu,\gamma}$ is bounded from $\mathcal{D}_{\alpha}$ into $\mathcal{D}_{\beta}$, we have that
\begin{align*}
\|\mathcal{H}_{\mu,\gamma}\|_{\mathcal{D}_{\alpha}\to\mathcal{D}_{\beta}}^2\|f_t\|_{\mathcal{D}_{\alpha}}^2&\geq\|\mathcal{H}_{\mu,\gamma}(f_t)\|_{{\mathcal D}_{\beta}}^2\\
&\succeq(1-t^2)^{2-\alpha}\sum_{n=0}^{\infty}(n+1)^{2\gamma-\beta-1}\left(\sum_{k=0}^{\infty}t^{n+2k}\mu([t,1))\right)^2\\
&\succeq(1-t^2)^{2-\alpha}\sum_{n=0}^{\infty}(n+1)^{2\gamma-\beta+1}t^{6n}\left(\mu([t,1))\right)^2\\
&\asymp\frac{(\mu([t,1))^2}{\ \ \ \ \ \ \ (1-t^2)^{2\gamma-\beta+\alpha}}.
\end{align*}
which implies $$\mu([t,1))\preceq(1-t^2)^{\gamma-\frac{\beta-\alpha}{2}},$$
hence $\mu$ is a $(\gamma-\frac{\beta-\alpha}{2})$-Carleson measure.
\end{proof}
\begin{lemma}\label{2.4}
Suppose that $0<\alpha\leq2$, $2\leq\beta<4$, $\gamma\geq1$ and let $\mu$ be a positive Borel measure on $[0,1)$ which satisfies the condition in Lemma $2.1$. Then the following two conditions are
equivalent:\\
\indent $(i)$  $\mu$ is a vanishing $(\gamma-\frac{\beta-\alpha}{2})$-Carleson measure.\\
\indent $(ii)$  $\mathcal{H}_{\mu,\gamma}$ is compact  from $\mathcal{D}_{\alpha}$ into $\mathcal{D}_{\beta}$.\\
\end{lemma}
\begin{proof}
$(i)\Rightarrow(ii)$. For any $f(z)=\sum_{n=0}^{\infty}a_nz^n\in\mathcal{D}_{\alpha}$. We define $S_{\mu,\gamma,m}$ as
$$S_{\mu,\gamma,m}(f)(z)=\sum_{n=0}^{m}\left(\sum_{k=0}^{\infty}(n+1)^{\frac{1-\beta}{2}}(k+1)^{\frac{\alpha-1}{2}}\mu_{n+k}a_k\right)\frac{\Gamma(n+\gamma)}{n!\Gamma(\gamma)}z^n.$$
Since $S_{\mu,\gamma,m}$ is finite rank operators, it is compact on $H^2$. Because $\mu$ is a vanishing $(\gamma-\frac{\beta-\alpha}{2})$-Carleson measure, applying Theorem $2.4$ of \cite{author7}, we have $$\mu_n=o\left(n^{-(2-\frac{\beta-\alpha}{2})}\right).\eqno{(2.4)}$$ For any $\epsilon>0$, there exists an $N>0$ such that when $m>N$, we have
$$|\mu_m|<\epsilon n^{-(2-\frac{\beta-\alpha}{2})}\eqno{(2.5)}.$$
Since
$$(S_{\mu,\gamma}-S_{\mu,\gamma,m})(f)(z)=\sum_{n=m+1}^{\infty}\left(\sum_{k=0}^{\infty}(n+1)^{\frac{1-\beta}{2}}(k+1)^{\frac{\alpha-1}{2}}\mu_{n+k}a_k\right)\frac{\Gamma(n+\gamma)}{n!\Gamma(\gamma)}z^n,$$
together with $(2.2)$ and $(2.5)$
we have

\begin{align*}
\|(S_{\mu,\gamma}-S_{\mu,\gamma,m})(f)(z)\|_{H^2}^{2}&\asymp\sum_{n=m+1}^{\infty}\left|\sum_{k=0}^{\infty}(n+1)^{\frac{2\gamma-\beta-1}{2}}(k+1)^{\frac{\alpha-1}{2}}\mu_{n,k}a_k\right|^2\\
&\preceq \epsilon^2\sum_{n=1}^{\infty}\left(\sum_{k=1}^{\infty}n^{\frac{2\gamma-\beta-1}{2}}k^{\frac{\alpha-1}{2}}\frac{|a_k|}{(n+k)^{\gamma-\frac{\beta-\alpha}{2}}}\right)^2\\
\end{align*}
From $(2.3)$, we obtain
$$\|(S_{\mu,\gamma}-S_{\mu,\gamma,m})(f)(z)\|_{H^2}^{2}\preceq\epsilon^2\|f\|_{H^2}^2.$$
Hence,
$$\|S_{\mu,\gamma}-S_{\mu,\gamma,m}\|_{H^2\to H^2}\preceq\epsilon.$$
 Which implies that $S_{\mu,\gamma}$ is compact on $H^2$, due to $$\mathcal{H}_{\mu,\gamma}=V_{\beta}^{-1}\circ S_{\mu,\gamma}\circ V_{\alpha},$$
  Since $V_{\beta}^{-1}$ and $V_{\alpha}$  are bounded,  $\mathcal{H}_{\mu,\gamma}$ is compact from $\mathcal{D}_{\alpha}$ into $\mathcal{D}_{\beta}$.\\
$(ii)\Rightarrow(i)$. Suppose that $\mathcal{H}_{\mu,\gamma}$ is compact from $\mathcal{D}_{\alpha}$ into $\mathcal{D}_{\beta}$.
For $0<t<1$, let $f_t(z)=(1-t^2)^{1-\frac{\alpha}{2}}\sum_{n=0}^{\infty}t^nz^n$, we have
$$\|f_t\|_{\mathcal{D}_{\alpha}}^2=(1-t^2)^{2-\alpha}\sum_{n=0}^{\infty}{(n+1)}^{1-\alpha}t^{2n}\asymp 1,$$
Since $\lim_{t\to 1}f_t(z)=0$ for any $z\in \mathbb{D}$, $f_t$ is convergent weakly to 0 in $\mathcal{D}_{\alpha}$ as $t\to 1$.
We have $$\lim_{t\to 1}\|\mathcal{H}_{\mu,\gamma}(f_t)\|_{{\mathcal D}_{\beta}}=0.$$
Similar to the proof of Lemma $2.3$, we have $$\mu([t,1)\preceq(1-t)^{\gamma-\frac{\beta-\alpha}{2}}\|\mathcal{H}_{\mu,\gamma}(f_t)\|_{\mathcal{D}_{\beta}}.$$
Hence
$$\lim_{t\to 1}\frac{\mu([t,1))}{\ \ \ \ \ \ \ (1-t)^{\gamma-\frac{\beta-\alpha}{2}}}=0,$$
$\mu$ is a vanishing $(\gamma-\frac{\beta-\alpha}{2})$-Carleson measure.

\end{proof}
\section*{Acknowledgments }
The first author was partially supported by Guangdong Basic and Applied Basic Research Foundation
(No. 2020B1515310001).


\begin{thebibliography}{999}
\bibitem{author7}
G. Bao, H. Wulan,
{Hankel matrices acting on Dirichlet spaces,} J. math. Anal. Appl., 409 (2014), 228-235.
\bibitem{author1}
P. L. Duren,
{Theory of $H^p$ Spaces},
Academic Press, New York, London, 1970. Reprint: Dover, Mineola, New York, 2000.
\bibitem{author9} P. Duren,
{Extension of a theorem of Carleson},
 Bull. Amer. Math. Soc., 75 (1969) 143-146. 
\bibitem{author6}
G. H. Hardy, J. E. Littlewood, G. Polya,
{Inequalities}, Cambridge Univ. Press, 1959.



\bibitem{author3}
Y. Xu, S. Ye, Z. Zhou,
{A Derivative-Hilbert operator acting on Dirichlet spaces,} arXiv:2207.08368 [math.FA].
\bibitem{author2}
S. Ye, Z. Zhou,
{Generalized Hilbert opertaor  acting on Bloch type spaces}, arXiv:2207.11170 [math.CV].
\bibitem{author4}
S. Ye, Z. Zhou,
{A Derivative-Hilbert operator acting on the Bloch space,} Complex Anal. Oper. Th., 15, (2021), 88.
\bibitem{author5}
S. Ye, Z. Zhou,
{A Derivative-Hilbert operator acting on the Bergman spaces,} J. Math. Anal. Appl.,  506 (2022), 125553.
\bibitem{author8}
K. Zhu,
{Operator Theory in Function spaces}, American Mathematical Soc., 2007.

\end{thebibliography}
\end{document}